pdfoutput=1
\documentclass[preprint,12pt]{elsarticle}

\usepackage{graphicx}
\usepackage{subcaption}
\usepackage{amssymb, amsmath}
\usepackage{amsthm}
\usepackage{color}
\usepackage{array}
\usepackage{tabu}

\journal{Russian Journal of Nonlinear Dynamics}

\newtheorem{theorem}{Theorem}[section]
\newtheorem{lemma}[theorem]{Lemma}

\theoremstyle{definition}

\theoremstyle{remark}

\begin{document}

\begin{frontmatter}


\title{Remarks on forced oscillations in some systems with gyroscopic forces}



\author[label5]{Ivan Polekhin}
\address[label5]{Steklov Mathematical Institute of the Russian Academy of Sciences, Moscow, Russia}


\begin{abstract}
In the paper we study the existence of a forced oscillation in two Lagrange systems with gyroscopic forces: a spherical pendulum in a magnetic field and a point on a rotating closed convex surface. We show how it is possible to prove the existence of forced oscillations in these systems provided the systems move in the presence of viscous friction.
\end{abstract}


\end{frontmatter}



\section{Introduction}
In \cite{bolotin2015calculus} the following result concerning the existence of forced oscillations in Lagrange systems has been proved. Let us have a system with Lagrangian $L$:
\begin{align*}
    L = L_2(q, \dot q, t) + L_1(q, \dot q, t) + L_0(q,t)
\end{align*}
defined on a smooth manifold $M$. Here, as usual, $L_2$ is a positively defined quadratic form in $\dot q$, $L_1$ is defined by a $1$-form, $L_0$ does not depend on $\dot q$. Let us have a submanifold with a smooth boundary $N \subset M$ defined by a smooth function $f$ and inequality $f \geqslant 0$, i.e. boundary $\partial N$ is defined by $f = 0$. We suppose that $L$ is $\tau$-periodic in $t$.

We will say that $N$ is dynamically convex (here we use the terminology from \cite{bolotin2015calculus}) if for any solution $q(t)$ tangent to the boundary $\partial N$ (in other words, $q(0) \in \partial N$ and $\dot q(0) \in T(\partial N)$) we have $f(q(t))<0$ for all $t \in (0, \varepsilon)$ for some $\varepsilon > 0$. This property plays the major role in all considerations below.

If $N$ is dynamically convex, then in each homotopy class of free closed loops in $N$, there exists the trajectory of a $\tau$-periodic solution of the Lagrange equations with Lagrangian $L$.

In particular, from this result it is possible to prove that for an inverted spherical pendulum with a horizontally $\tau$-periodically moving pivot point there always exists a $\tau$-periodic solution along which the rod of pendulum is always above the horizontal plane.

The result is formally applicable to systems with gyroscopic forces. These forces correspond to the linear in $\dot q$ in the Lagrangian function. Gyroscopic forces naturally appear in various mechanical system. For instance, they appear after Routh reduction in systems with symmetry, in the presence of magnetic forces (the Lorentz force), or due to the use of a rotating reference frame.

As we will show below, in practice, the presence of gyroscopic forces may cause difficulties to the application of the result \cite{bolotin2015calculus}. The main difficulty is that, in the presence of gyroscopic forces, $N$ can be dynamically convex only in special cases. To illustrate this, let us consider an inverted spherical pendulum in a magnetic field and in a horizontal $\tau$-periodic force field.

The equation of motion has the form
\begin{align}
\label{e2}
    m\ddot \rho = F + R + [\dot \rho, B] - mg \cdot e_z,
\end{align}
where $m$ is the mass of the pendulum, $F$ is the horizontal vector of external forces, $R$ is the force of constraint, $B$ is the magnetic field, $g$ is the gravity acceleration. Here $e_x$, $e_y$ and $e_z$ are, as usual, orthogonal unit vectors. $F = F_x(t) e_x + F_y(t) e_y$, $F_x(t+\tau) = F_x(t)$ and $F_y(t+\tau) = F_y(t)$.

If $B \equiv 0$, then
$$
N = \{ \rho \colon (\rho, e_z) \geqslant 0 \}
$$
is a dynamically convex region. Indeed, for any solution $\rho(t) = (\rho_x(t), \rho_y(t), \rho_z(t))$ such that $\rho_z(0) = 0$ and $(\dot\rho(0), e_z) = 0$ we have
$$
(m\ddot\rho, e_z) = -mg < 0.
$$
Therefore, from the Taylor expansion we immediately obtain that $N$ is dynamically convex (Fig. 1).

If now $B$ is a non-zero constant vector, we have
$$
(m\ddot\rho(0), e_z) = -mg + ([\dot\rho(0), B], e_z).
$$
Therefore, we see that for large $\dot\rho(0)$ right hand side can be both positive and negative and $N$ is not dynamically convex (in a general case). To be more precise, if $|\dot \rho(0)|$ is large and $(m\ddot\rho(0), e_z) < 0$, then we can change the direction of the velocity and the considered value becomes positive. The only case when $N$ is still dynamically convex is when $B$ and $e_z$ are parallel and, therefore, $([\dot\rho(0), B], e_z) = 0$ (Fig. 2).

Below we will show that it is possible to prove the existence of a periodic solution for equation (\ref{e2}) provided the pendulum is moving with viscous friction. Note that the original result in \cite{bolotin2015calculus} cannot be applied to systems with friction.

The results are based on the ideas of the Wa{\.z}ewski topological method \cite{wazewski1947principe,reissig1963qualitative} and on more general results on the existence of forced oscillations \cite{srzednicki1994periodic,srzednicki2005fixed}. The paper can be considered as a continuation of the previous results \cite{polekhin2014examples,polekhin2014periodic,polekhin2015forced,polekhin2016forced,polekhin2018topological,srzednicki2019periodic}.

\begin{figure}[!h]
\centering
\begin{minipage}[t]{180px}
  \centering
  \def\svgwidth{180px}\footnotesize
\begingroup%
  \makeatletter%
  \providecommand\color[2][]{%
    \errmessage{(Inkscape) Color is used for the text in Inkscape, but the package 'color.sty' is not loaded}%
    \renewcommand\color[2][]{}%
  }%
  \providecommand\transparent[1]{%
    \errmessage{(Inkscape) Transparency is used (non-zero) for the text in Inkscape, but the package 'transparent.sty' is not loaded}%
    \renewcommand\transparent[1]{}%
  }%
  \providecommand\rotatebox[2]{#2}%
  \newcommand*\fsize{\dimexpr\f@size pt\relax}%
  \newcommand*\lineheight[1]{\fontsize{\fsize}{#1\fsize}\selectfont}%
  \ifx\svgwidth\undefined%
    \setlength{\unitlength}{302.10685959bp}%
    \ifx\svgscale\undefined%
      \relax%
    \else%
      \setlength{\unitlength}{\unitlength * \real{\svgscale}}%
    \fi%
  \else%
    \setlength{\unitlength}{\svgwidth}%
  \fi%
  \global\let\svgwidth\undefined%
  \global\let\svgscale\undefined%
  \makeatother%
  \begin{picture}(1,0.75308518)%
    \lineheight{1}%
    \setlength\tabcolsep{0pt}%
    \put(0,0){\includegraphics[width=\unitlength,page=1]{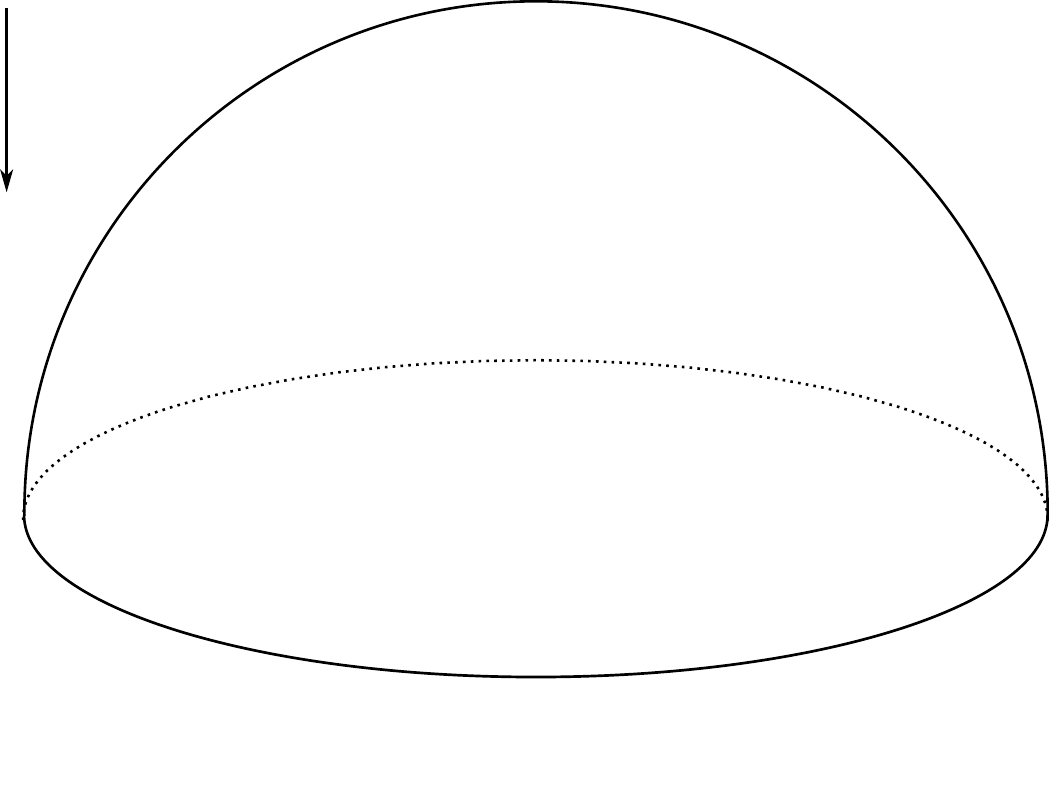}}%
    \put(0.01633465,0.64239797){\color[rgb]{0,0,0}\makebox(0,0)[lt]{\lineheight{1.25}\smash{\begin{tabular}[t]{l}$mg$\end{tabular}}}}%
    \put(0.72227175,0.42808407){\color[rgb]{0,0,0}\makebox(0,0)[lt]{\lineheight{1.25}\smash{\begin{tabular}[t]{l}$T=c$\end{tabular}}}}%
    \put(0,0){\includegraphics[width=\unitlength,page=2]{drawing-1.pdf}}%
    \put(0.69354227,0.05163427){\color[rgb]{0,0,0}\makebox(0,0)[lt]{\lineheight{1.25}\smash{\begin{tabular}[t]{l}$\dot\rho$\end{tabular}}}}%
  \end{picture}%
\endgroup%

    \caption{For $B\equiv0$ the region is dynamically convex: tangent solutions locally leaves the region. In the presence of friction, Solutions cannot leave the subset where $T \leqslant c$ for some $c$.}
    \label{pic3}
\end{minipage}
\hspace{0.1cm}
\begin{minipage}[t]{180px}
  \centering
  \def\svgwidth{180px}\footnotesize
\begingroup%
  \makeatletter%
  \providecommand\color[2][]{%
    \errmessage{(Inkscape) Color is used for the text in Inkscape, but the package 'color.sty' is not loaded}%
    \renewcommand\color[2][]{}%
  }%
  \providecommand\transparent[1]{%
    \errmessage{(Inkscape) Transparency is used (non-zero) for the text in Inkscape, but the package 'transparent.sty' is not loaded}%
    \renewcommand\transparent[1]{}%
  }%
  \providecommand\rotatebox[2]{#2}%
  \newcommand*\fsize{\dimexpr\f@size pt\relax}%
  \newcommand*\lineheight[1]{\fontsize{\fsize}{#1\fsize}\selectfont}%
  \ifx\svgwidth\undefined%
    \setlength{\unitlength}{302.10685959bp}%
    \ifx\svgscale\undefined%
      \relax%
    \else%
      \setlength{\unitlength}{\unitlength * \real{\svgscale}}%
    \fi%
  \else%
    \setlength{\unitlength}{\svgwidth}%
  \fi%
  \global\let\svgwidth\undefined%
  \global\let\svgscale\undefined%
  \makeatother%
  \begin{picture}(1,0.76165325)%
    \lineheight{1}%
    \setlength\tabcolsep{0pt}%
    \put(0,0){\includegraphics[width=\unitlength,page=1]{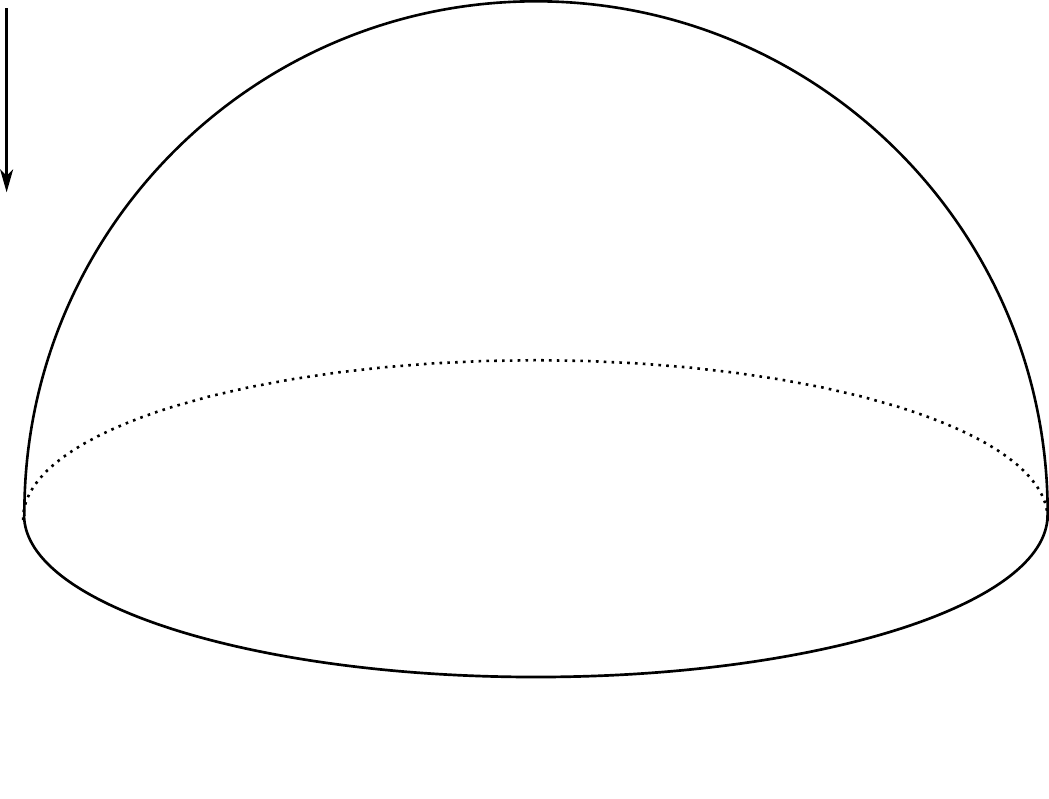}}%
    \put(0.01633465,0.65096603){\color[rgb]{0,0,0}\makebox(0,0)[lt]{\lineheight{1.25}\smash{\begin{tabular}[t]{l}$mg$\end{tabular}}}}%
    \put(0.72227175,0.43665214){\color[rgb]{0,0,0}\makebox(0,0)[lt]{\lineheight{1.25}\smash{\begin{tabular}[t]{l}$T=c$\end{tabular}}}}%
    \put(0,0){\includegraphics[width=\unitlength,page=2]{drawing-2.pdf}}%
    \put(0.69354227,0.06020233){\color[rgb]{0,0,0}\makebox(0,0)[lt]{\lineheight{1.25}\smash{\begin{tabular}[t]{l}$\dot\rho$\end{tabular}}}}%
    \put(0,0){\includegraphics[width=\unitlength,page=3]{drawing-2.pdf}}%
    \put(0.85257486,0.63842717){\color[rgb]{0,0,0}\makebox(0,0)[lt]{\lineheight{1.25}\smash{\begin{tabular}[t]{l}$B$\end{tabular}}}}%
    \put(0,0){\includegraphics[width=\unitlength,page=4]{drawing-2.pdf}}%
    \put(0.31784784,0.05819707){\color[rgb]{0,0,0}\makebox(0,0)[lt]{\lineheight{1.25}\smash{\begin{tabular}[t]{l}$\dot\rho$\end{tabular}}}}%
  \end{picture}%
\endgroup%

    \caption{For $B \ne 0$ the region can be non-convex.}
    \label{pic4}
\end{minipage}
\end{figure}

\section{Spherical pendulum in a magnetic field}

The equation of motion of a spherical pendulum with viscous friction in a magnetic field has the form
\begin{align}
\label{ee2}
    m\ddot \rho = -\mu \dot \rho + F + R + [\dot \rho, B] - mg \cdot e_z,
\end{align}
where, $\mu > 0$ is the friction coefficient. Again, $F$ and $B$ are $\tau$-periodic functions. Here and below we use this form of equations and do not use Lagrange equations. However, it is always possible to rewrite the above equation in local coordinates.

Let
$$
T = \frac{m}{2}(\dot\rho, \dot\rho)
$$
be the kinetic energy of the system. The phase space of our system is $T\mathbb{S}^2$. Let us consider a submanifold $K$ (with boundary) of $T\mathbb{S}^2$ defined by the following equation: $T \leqslant c$. In other words, for all points of $\mathbb{S}^2$ we consider corresponding tangent spaces and in each tangent space we consider only relatively small velocities.

If $c > 0$ is large enough, then for any solution starting at $\partial K$ we have
\begin{align}
\label{ee3}
    \frac{d}{dt} T < 0.
\end{align}
Indeed,
$$
\frac{d}{dt} T = (m\ddot\rho, \dot\rho) = -\mu(\dot\rho, \dot\rho) + (F + R + [\dot\rho, B], \dot\rho) = -\mu(\dot\rho, \dot\rho) + (F,\dot\rho)< 0,
$$
provided $c$ is large.

Below we will use without a proof a result from \cite{srzednicki1994periodic} and some definitions will be needed.

Let us have a closed set $N \subset \mathbb{R} \times T \mathbb{S}^2$ (extended phase space) satisfying the following properties. The boundary of $N$ is smooth. $N$ is $\tau$-periodic in $t$, i.e. the set coincides with itself after the translation $t \mapsto t + \tau$. The point $(t_0, \rho_0, \dot \rho_0) \in \partial N$ is an egress point of $N$ w.r.t. system (\ref{ee2}) if for some $\varepsilon$ we have $(t,\rho(t),\dot\rho(t)) \in N \setminus \partial N$ for all $t \in (-\varepsilon + t_0, t_0)$, where $\rho(t)$ is a solution of (\ref{ee2}) such that $\rho(t_0) = \rho_0$ and $\dot\rho(t_0) = \dot\rho_0$. Similarly, we say that $(t_0, \rho_0, \dot \rho_0) \in \partial N$ is a strictly egress point if for some $\varepsilon > 0$ we have $(t,\rho(t),\dot\rho(t)) \not\in N$ for all $t \in (t_0, t_0 + \varepsilon)$.

Now we can prove the following result
\begin{theorem}
Let us consider the following equation of motion of a spherical pendulum in a magnetic field in the presence of viscous friction and an additional horizontal force
\begin{align}
\label{e4}
    m\ddot \rho = -\mu \dot \rho + F_x(t) e_x + F_y(t) e_y + R + [\dot \rho, B(t)] - mg \cdot e_z,
\end{align}
where $F_x$, $F_y$ and $B$ are $\tau$-periodic functions of time. Let $c > 0$ be a constant such that for any $\dot\rho \,\bot\, e_z$ and $m\dot\rho^2/2 \leqslant c$ we have for all $t$
\begin{align}
\label{eq5}
    ([\dot \rho, B(t)], e_z) < mg.
\end{align}
Then for any
\begin{align}
\label{e6}
    \mu > \max\limits_{t \in [0,\tau]} \sqrt{\frac{m}{2c}}(|F| + mg)
\end{align}
there exists a $\tau$-periodic solution $\rho(t)$ such that $(\rho(t),e_z) > 0$ for all $t$.
\end{theorem}
\begin{proof}
First, let us mention that any solution of the considered equation can be continued for all $t$. This follows from inequality (\ref{ee3}): the configuration space is compact, the forces are periodic and the velocities are bounded.

Let us now consider the following subset $N_c$ of the extended phase space:
\begin{align}
    N_c = \{ t, \rho, \dot\rho \colon (\rho, e_z) \geqslant 0, (\dot \rho, \dot \rho) \leqslant 2c/m \}.
\end{align}
The boundary of $N_c$ is the following set
\begin{align}
    \partial N_c = \{ t, \rho, \dot\rho \colon (\rho, e_z) = 0, (\dot \rho, \dot \rho) \leqslant 2c/m \} \cup \{ t, \rho, \dot\rho \colon (\rho, e_z) > 0, (\dot \rho, \dot \rho) = 2c/m \}.
\end{align}
From (\ref{eq5}) we have that the set of strictly egress point has the form
$$
N_c^{++} = \partial N_c \cap \{ t, \rho, \dot \rho \colon (\rho, e_z) = 0, \, (\dot\rho, e_z) \leqslant 0 \}.
$$
The set of strictly egress points coincides is closed. $N_c$ is homotopic to a point and $N_c^{++}$ is homotopic to a circle, i.e. $\chi(N_c) - \chi(N_c^{++}) = 1$. Using the result from \cite{srzednicki1994periodic,srzednicki2005fixed}, we obtain that in $N_c \setminus \partial N_c$ there exists a $\tau$-periodic solution. In particular, we have shown that there exists a solution such that the rod of the pendulum is never horizontal for all $t$.
\end{proof}

We do not present the theorem from \cite{srzednicki1994periodic,srzednicki2005fixed} and refer the reader to the original papers. However, we would like to outline the idea of the theorem.

The general idea is to choose an appropriate closed periodic subset $N$ of the extended phase space and to consider the vector field in a vicinity of its boundary. All points of the boundary can be divided into two classes (possibly empty). For the first class, the corresponding trajectory (starting at the point and considered for $t \geqslant t_0$) intersects the boundary transversely and for the second class the trajectory is tangent to the boundary. We suppose that the boundary is smooth. Each of these classes can be also divided into two sets: each transverse trajectory can either leave the set or enter it. Any tangent trajectory is either externally tangent or internally tangent to the boundary.

The key requirement is that there should be no trajectories that are internally tangent to our region, i.e. any egress point is a strictly egress point. In particular, for our system this fact follows from inequality (\ref{eq5}). Then we call all points of the boundary, that correspond to externally tangent trajectories and to transverse trajectories leaving the set, the strictly egress points of our system.

Next step is to analyse the topology of the set of strictly egress points $N^{++}$. To be more precise, if this set is also periodic, then we should calculate the Euler characteristic of this set. If $\chi(N) - \chi(N^{++}) \ne 0$, then there is a $\tau$-periodic solution inside $N$.

As a simple example of these considerations we can consider the following $\tau$-periodic system.
\begin{align*}
    \begin{split}
       & \dot x = v(x, y, t),\\
       & \dot y = w(x, y, t).
    \end{split}
\end{align*}
Suppose that $N$ is the following cylinder in $\mathbb{R}^3$: $N = \{ x, y, t \colon x^2 + y^2 \leqslant 1 \}$ and for all $t$ all solutions are transverse to the boundary and, moreover, all corresponding solutions leave the cylinder (Fig. 3). Then, obviously $N^{++}$ is homotopic to a circle and $\chi(N^{++}) = 0$, $\chi(N) = 1$. Hence, there is a $\tau$-periodic solution inside $N$.

\begin{figure}[!h]
\centering
\begin{minipage}[t]{180px}
  \centering
  \def\svgwidth{180px}\footnotesize
\begingroup%
  \makeatletter%
  \providecommand\color[2][]{%
    \errmessage{(Inkscape) Color is used for the text in Inkscape, but the package 'color.sty' is not loaded}%
    \renewcommand\color[2][]{}%
  }%
  \providecommand\transparent[1]{%
    \errmessage{(Inkscape) Transparency is used (non-zero) for the text in Inkscape, but the package 'transparent.sty' is not loaded}%
    \renewcommand\transparent[1]{}%
  }%
  \providecommand\rotatebox[2]{#2}%
  \newcommand*\fsize{\dimexpr\f@size pt\relax}%
  \newcommand*\lineheight[1]{\fontsize{\fsize}{#1\fsize}\selectfont}%
  \ifx\svgwidth\undefined%
    \setlength{\unitlength}{305.42135827bp}%
    \ifx\svgscale\undefined%
      \relax%
    \else%
      \setlength{\unitlength}{\unitlength * \real{\svgscale}}%
    \fi%
  \else%
    \setlength{\unitlength}{\svgwidth}%
  \fi%
  \global\let\svgwidth\undefined%
  \global\let\svgscale\undefined%
  \makeatother%
  \begin{picture}(1,0.86542365)%
    \lineheight{1}%
    \setlength\tabcolsep{0pt}%
    \put(0,0){\includegraphics[width=\unitlength,page=1]{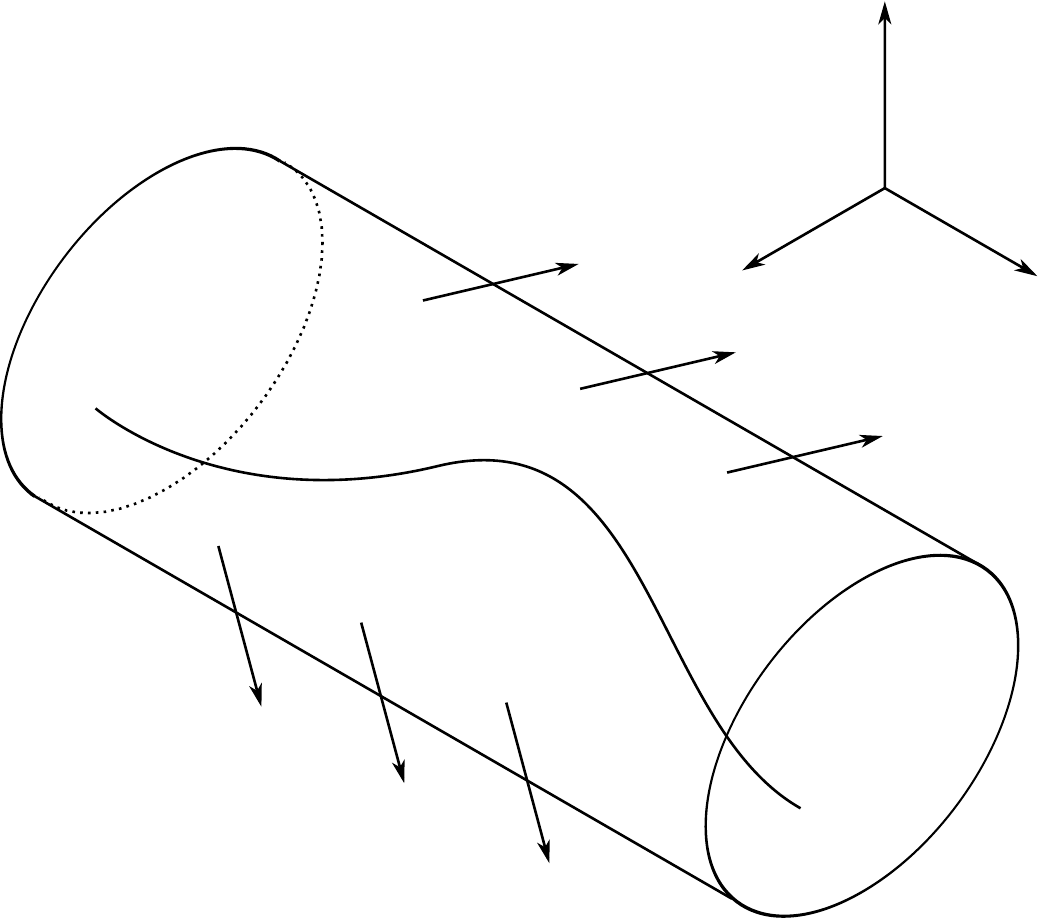}}%
    \put(0.84891865,0.82763006){\color[rgb]{0,0,0}\makebox(0,0)[lt]{\lineheight{1.25}\smash{\begin{tabular}[t]{l}$y$\end{tabular}}}}%
    \put(0.73578466,0.58732976){\color[rgb]{0,0,0}\makebox(0,0)[lt]{\lineheight{1.25}\smash{\begin{tabular}[t]{l}$x$\end{tabular}}}}%
    \put(0.92153514,0.56680777){\color[rgb]{0,0,0}\makebox(0,0)[lt]{\lineheight{1.25}\smash{\begin{tabular}[t]{l}$t$\end{tabular}}}}%
  \end{picture}%
\endgroup%

    \caption{A periodic solution for a planar non-autonomous system. The cylinder $N$ is shown.}
    \label{pic3}
\end{minipage}
\hspace{0.1cm}
\begin{minipage}[t]{180px}
  \centering
  \def\svgwidth{180px}\footnotesize
\begingroup%
  \makeatletter%
  \providecommand\color[2][]{%
    \errmessage{(Inkscape) Color is used for the text in Inkscape, but the package 'color.sty' is not loaded}%
    \renewcommand\color[2][]{}%
  }%
  \providecommand\transparent[1]{%
    \errmessage{(Inkscape) Transparency is used (non-zero) for the text in Inkscape, but the package 'transparent.sty' is not loaded}%
    \renewcommand\transparent[1]{}%
  }%
  \providecommand\rotatebox[2]{#2}%
  \newcommand*\fsize{\dimexpr\f@size pt\relax}%
  \newcommand*\lineheight[1]{\fontsize{\fsize}{#1\fsize}\selectfont}%
  \ifx\svgwidth\undefined%
    \setlength{\unitlength}{305.42134743bp}%
    \ifx\svgscale\undefined%
      \relax%
    \else%
      \setlength{\unitlength}{\unitlength * \real{\svgscale}}%
    \fi%
  \else%
    \setlength{\unitlength}{\svgwidth}%
  \fi%
  \global\let\svgwidth\undefined%
  \global\let\svgscale\undefined%
  \makeatother%
  \begin{picture}(1,0.86542369)%
    \lineheight{1}%
    \setlength\tabcolsep{0pt}%
    \put(0,0){\includegraphics[width=\unitlength,page=1]{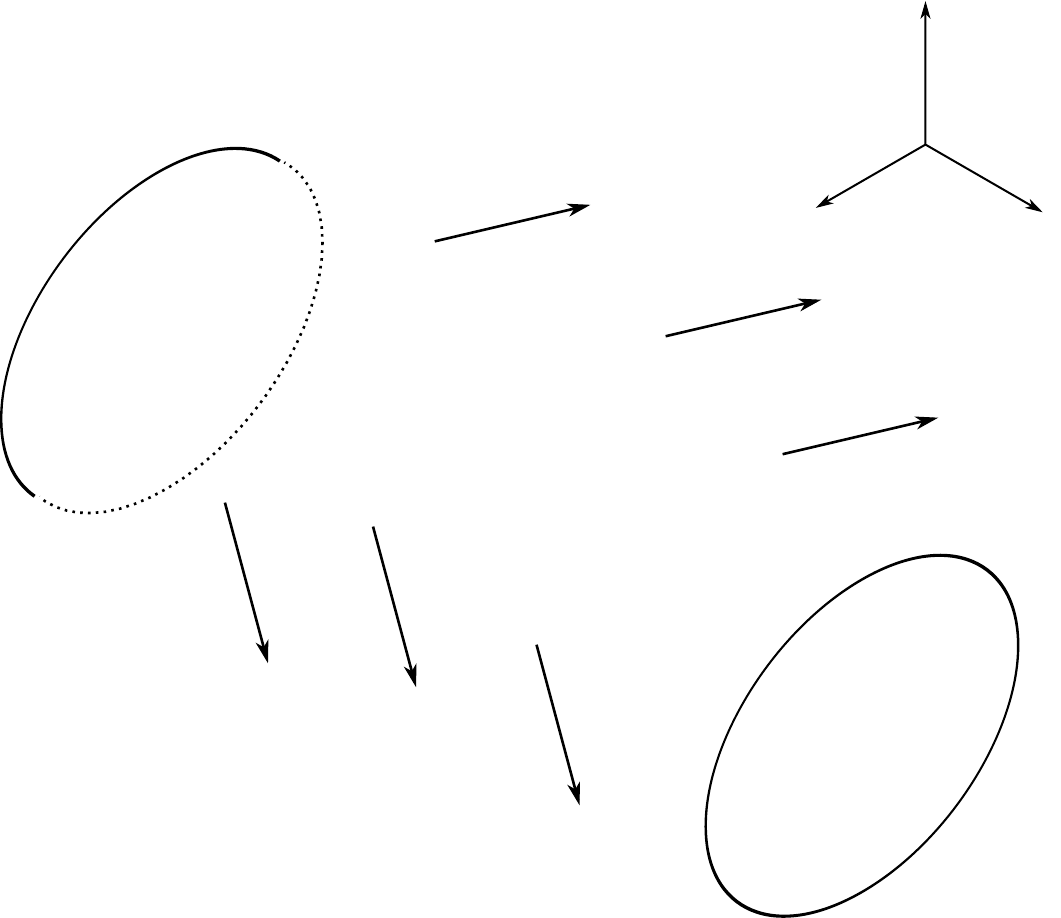}}%
    \put(0.88615049,0.8363752){\color[rgb]{0,0,0}\makebox(0,0)[lt]{\lineheight{1.25}\smash{\begin{tabular}[t]{l}$y$\end{tabular}}}}%
    \put(0.79919457,0.65167824){\color[rgb]{0,0,0}\makebox(0,0)[lt]{\lineheight{1.25}\smash{\begin{tabular}[t]{l}$x$\end{tabular}}}}%
    \put(0.93969121,0.62721877){\color[rgb]{0,0,0}\makebox(0,0)[lt]{\lineheight{1.25}\smash{\begin{tabular}[t]{l}$t$\end{tabular}}}}%
    \put(0,0){\includegraphics[width=\unitlength,page=2]{drawing2.pdf}}%
  \end{picture}%
\endgroup%

    \caption{A periodic solution in a continuously deformed cylinder.}
    \label{pic4}
\end{minipage}
\end{figure}

The following proposition can be proved without involving any auxiliary results.

\begin{theorem}
Suppose that in (\ref{e4}) $F_x$, $F_y$ and $B$ are bounded functions of time and condition (\ref{eq5}) holds. Then for any
\begin{align}
    \mu > \sup\limits_{t \in \mathbb{R}} \sqrt{\frac{m}{2c}}(|F| + mg)
\end{align}
there exists a solution $\rho(t)$ such that $(\rho(t),e_z) > 0$ for all $t$.
\end{theorem}
\begin{proof}
The idea is of the proof stems from the Wa{\.z}ewski topological method \cite{wazewski1947principe,reissig1963qualitative} that can be explained as follows. Let us consider a closed subset $D$ of $N_c$ such that
\begin{enumerate}
    \item $D$ is homeomorphic to a disk.
    \item $\partial D \subset N_c^{++}$ and $D \cap \partial N_c = \partial D$.
    \item $D \subset \{ t, \rho, \dot \rho \colon t = 0 \}$.
\end{enumerate}
Let us now show that there exists a solution $\rho(t)$ starting at $D$ such that $(\rho(t),e_z) > 0$. Suppose the contrary, i.e. for any point $t = 0$, $\rho_0$, $\dot\rho_0$ the corresponding solution $\rho(t)$ ($\rho(0) = \rho_0$ and $\dot\rho(0)=\dot\rho_0$) leaves $N_c$ at time $t'$. In other words, point $(t', \rho(t'), \dot\rho(t')) \in N_c^{++}$. Since all egress points are strictly egress points, then the corresponding map from $D$ to $\partial N_c$ is continuous: for any point $t = 0$, $\rho'_0$, $\dot\rho'_0$ close to $t = 0$, $\rho_0$, $\dot\rho_0$, the corresponding egress point  $(t'', \rho'(t''), \dot\rho'(t'')) \in N_c^{++}$ is close to $(t', \rho(t'), \dot\rho(t'))$. Therefore, one can construct a retraction of the disk onto its boundary. The contradiction proves the theorem.
\end{proof}

\section{Rotating closed surface}
Let us consider a closed strictly convex surface in $\mathbb{R}^3$. Suppose that this surface is rotating around a fixed point w.r.t. a prescribed law of motion. Let there be a mass point moving with viscous friction on this surface. In this section we show that if the rotation is slow and smooth enough and the friction is strong enough, then there exist a solution along which the mass point always remains `in the top half' of the rotating surface. More precise definitions and results are given below.

The result is trivial if the surface is motionless: there exists the unstable equilibrium at the top of the surface. We show that if we rotate periodically our surface in the presence of friction, then there also exists a periodic solution

Let $Oxyz$ be the fixed frame, the force of gravity is directed along $z$-axis: $F_{gravity} = -mg \cdot e_z$, where $g$ is the gravity acceleration and $m$ is the mass of the point. Let $O\xi\eta\zeta$ be the rotating frame and the surface can be described in these coordinates by the following equation:
\begin{align}
\label{e9}
    s(\xi, \eta, \zeta) = 0.
\end{align}
Where $s$ is a smooth function. We suppose that this level set is closed and strictly convex. We also suppose that the point $O$ (the centre of rotation) is inside the region bounded by the surface.

\begin{figure}[!h]
\centering
\begin{minipage}[t]{180px}
  \centering
  \def\svgwidth{180px}\footnotesize
  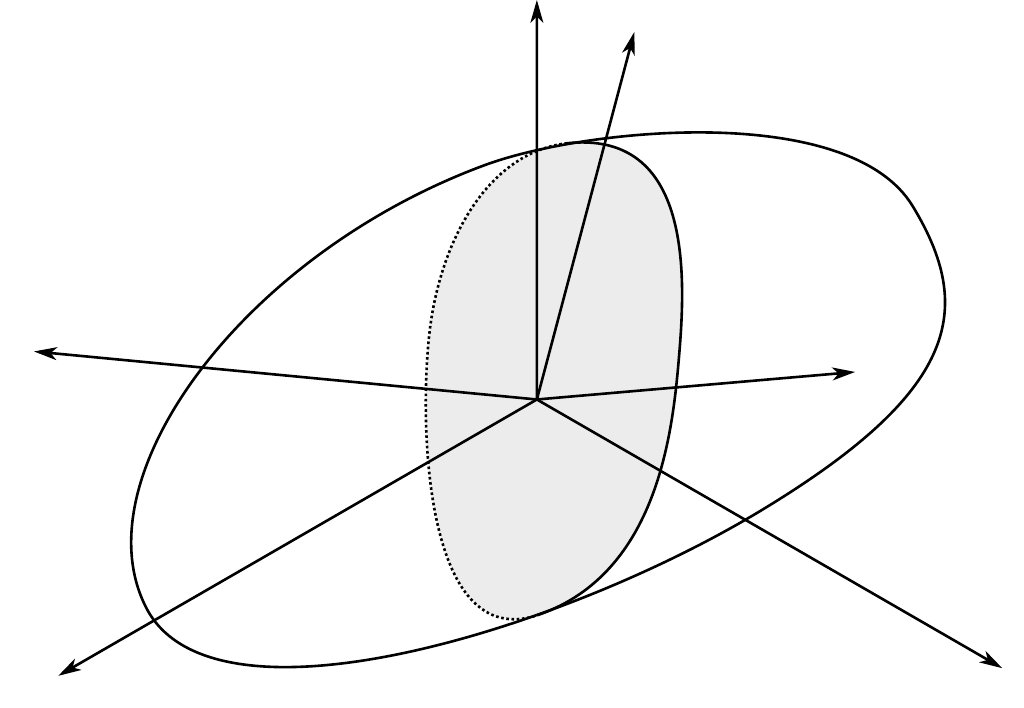
    \caption{Schematic presentation of the rotating surface in a gravitational field.}
    \label{pic3}
\end{minipage}
\hspace{0.1cm}
\begin{minipage}[t]{180px}
  \centering
  \def\svgwidth{180px}\footnotesize
\begingroup%
  \makeatletter%
  \providecommand\color[2][]{%
    \errmessage{(Inkscape) Color is used for the text in Inkscape, but the package 'color.sty' is not loaded}%
    \renewcommand\color[2][]{}%
  }%
  \providecommand\transparent[1]{%
    \errmessage{(Inkscape) Transparency is used (non-zero) for the text in Inkscape, but the package 'transparent.sty' is not loaded}%
    \renewcommand\transparent[1]{}%
  }%
  \providecommand\rotatebox[2]{#2}%
  \newcommand*\fsize{\dimexpr\f@size pt\relax}%
  \newcommand*\lineheight[1]{\fontsize{\fsize}{#1\fsize}\selectfont}%
  \ifx\svgwidth\undefined%
    \setlength{\unitlength}{295.00557486bp}%
    \ifx\svgscale\undefined%
      \relax%
    \else%
      \setlength{\unitlength}{\unitlength * \real{\svgscale}}%
    \fi%
  \else%
    \setlength{\unitlength}{\svgwidth}%
  \fi%
  \global\let\svgwidth\undefined%
  \global\let\svgscale\undefined%
  \makeatother%
  \begin{picture}(1,0.6901277)%
    \lineheight{1}%
    \setlength\tabcolsep{0pt}%
    \put(0,0){\includegraphics[width=\unitlength,page=1]{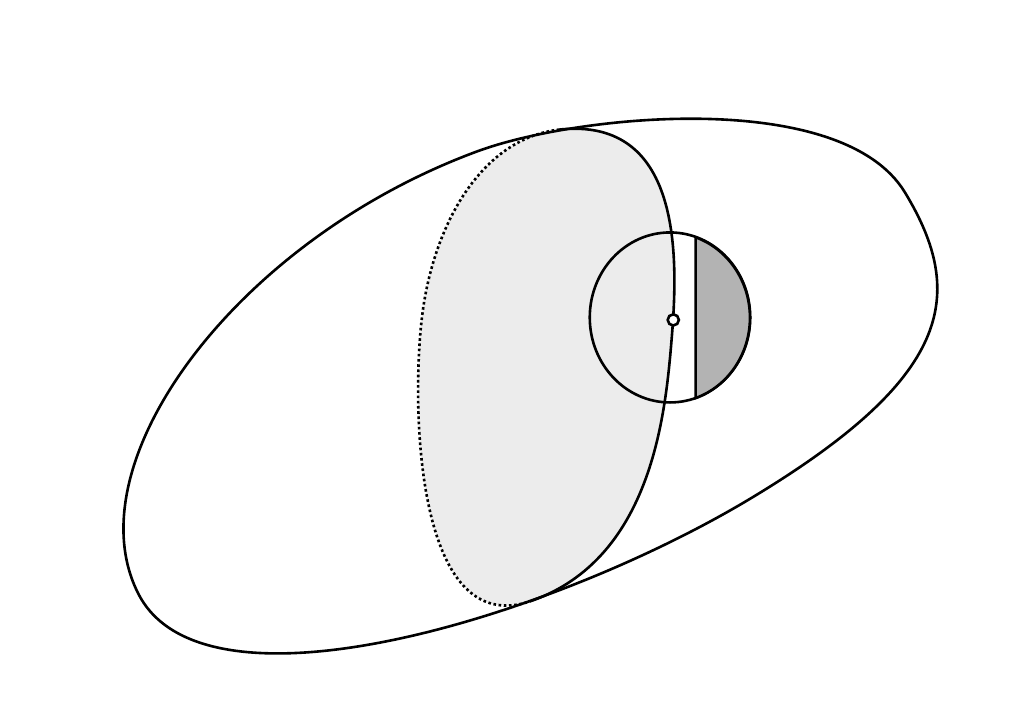}}%
    \put(0.48086666,0.22735764){\color[rgb]{0,0,0}\makebox(0,0)[lt]{\lineheight{1.25}\smash{\begin{tabular}[t]{l}$T = c$\end{tabular}}}}%
    \put(0,0){\includegraphics[width=\unitlength,page=2]{surf2.pdf}}%
    \put(0.72886348,0.49054091){\color[rgb]{0,0,0}\makebox(0,0)[lt]{\lineheight{1.25}\smash{\begin{tabular}[t]{l}$T \leqslant c \cap \dot f \leqslant 0$\end{tabular}}}}%
    \put(0,0){\includegraphics[width=\unitlength,page=3]{surf2.pdf}}%
  \end{picture}%
\endgroup%

    \caption{For a given $\rho$, the set of strictly egress points in the tangent space is closed (highlighted). We suppose that $T \leqslant c$.}
    \label{pic4}
\end{minipage}
\end{figure}

The equation of motion in the rotating frame has the following form:
\begin{align}
\label{e10}
    m\ddot \rho = -mg\cdot e_z + R - m [\dot\omega, \rho] - m [\omega, [\omega, \rho]] -2m [\omega, \dot\rho] - \mu\dot\rho.
\end{align}
Here $R$ is the force of reaction directed along the normal vector to the surface (at a given point), $\rho$ is the radius-vector of the mass point in the rotating frame,
$$
\rho = \xi e_\xi + \eta e_\eta + \zeta e_\zeta, \quad \dot \rho = \dot\xi e_\xi + \dot\eta e_\eta + \dot\zeta e_\zeta, \quad \ddot \rho = \ddot\xi e_\xi + \ddot\eta e_\eta + \ddot\zeta e_\zeta.
$$
$\omega$ is the vector of angular velocity for rotation of $O\xi\eta\zeta$ w.r.t. $Oxyz$
$$
\omega = \omega_\xi e_\xi + \omega_\eta e_\eta + \omega_\zeta e_\zeta, \quad \dot\omega = \dot\omega_\xi e_\xi + \dot\omega_\eta e_\eta + \dot\omega_\zeta e_\zeta.
$$
$\mu > 0$ is the coefficient of the viscous friction in the system.
The phase space of the system is $T\mathbb{S}^2$. Let us consider the horizontal plane
$$
z = 0.
$$
In the moving frame, this plane is described by the equation
\begin{align}
\label{e11}
    f(t, \xi, \eta, \zeta) = 0,
\end{align}
where
$$
f = (e_z, e_\xi) \xi + (e_z, e_\eta) \eta + (e_z, e_\zeta) \zeta.
$$
For any $t$, this plane intersects the surface $s = 0$ by some curve homotopic to a circle. The form of this curve smoothly depends on $t$. Let us also consider the following equations
\begin{align}
\label{e12}
    \dot s = \frac{\partial g}{\partial \xi} \dot \xi + \frac{\partial g}{\partial \eta} \dot \eta + \frac{\partial g}{\partial \zeta} \dot \zeta = 0,
\end{align}
and
\begin{align}
\label{e13}
\begin{split}
    \dot f = (e_z, [\omega, e_\xi])\xi + (e_z, [\omega, e_\eta])\eta &+ (e_z, [\omega, e_\zeta])\zeta +\\
    &+(e_z, e_\xi) \dot \xi + (e_z, e_\eta) \dot \eta + (e_z, e_\zeta) \dot \zeta = 0.
\end{split}
\end{align}
If, for given $\xi$, $\eta$ and $\zeta$, tangent vector $(\dot \xi, \dot \eta, \dot \zeta)$ satisfies equations (\ref{e12}) and (\ref{e13}), then, in the first order in $t$, the solution, starting in the considered point with the considered velocity, stays on the surface $f = 0$.

Below we show that if $\mu > 0$ is large and $\omega$ and $\dot \omega$ are small enough, then in the second order in $t$ every solution which is tangent to $f = s = 0$ leaves the region where $f \geqslant 0$. To be more precise:
\begin{lemma}
Let us consider the kinetic energy of the system:
\begin{align}
    T = \frac{m\dot\rho^2}{2}.
\end{align}
Suppose that $\omega$ and $\dot\omega$ are bounded and $\mu > 0$. Then there exists $c$ such that for any solution $\rho(t)$ of (\ref{e10}) satisfying $m\dot\rho_0^2 = m\dot\rho^2(t_0) = 2c$, we have $\dot T < 0$ at $t = t_0$.
\end{lemma}
\begin{proof}
Indeed,
$$
\frac{d}{dt}T = m(\ddot\rho, \dot \rho).
$$
Therefore, taking into account that $R\,\bot\,\dot \rho$, we obtain
$$
\frac{d}{dt}T =  -(mg\cdot e_z, \dot \rho) -m([\dot\omega, \rho], \dot \rho) - m ([\omega, [\omega, \rho]], \dot \rho) - \mu(\dot\rho, \dot\rho).
$$
Finally, if $c$ is large enough, then for any $\dot\rho$ such that $T = c$, we have $\dot T < 0$. It proves the lemma.
\end{proof}
Note, that if $\mu \to \infty$, then $c$ can be chosen arbitrarily small.
\begin{lemma}
There exist $b$ and $c$ such that for any $\xi$, $\eta$, $\zeta$, $\dot\xi$, $\dot \eta$ and $\dot\zeta$ satisfying (\ref{e9}), (\ref{e11})--(\ref{e13}), we have
$$
\ddot f < 0.
$$
provided $|\omega| < b$, $|\dot\omega| < b$ and $m\dot\rho^2 \leqslant 2c$.
\end{lemma}
\begin{proof}
First, suppose that in equation (\ref{e9}) variables $\xi$ and $\eta$ can be locally considered as independent, i.e. we can rewrite (\ref{e9}) locally as $$\zeta - \zeta(\xi, \eta) = 0.$$
Where $\zeta(\xi, \eta)$ is a smooth function. From (\ref{e10}) we have
\begin{align}
\label{e15}
    \begin{split}
        & \ddot \xi = -mg(e_z,e_\xi) + \Phi_\xi + \tilde R\cdot \left( -\zeta_\xi \right),\\
        & \ddot \eta = -mg(e_z,e_\eta) + \Phi_\eta + \tilde R \cdot \left( -\zeta_\eta \right),\\
        & \ddot \zeta = -mg(e_z,e_\zeta) + \Phi_\zeta + \tilde R.
    \end{split}
\end{align}
Here $\Phi$ is the vector of all forces of inertia and friction acting on the point in the rotating frame and $R = \tilde R (-\zeta_\xi e_\xi - \zeta_\eta e_\eta + e_\zeta)$. It is not hard to see that $\| \Phi \| \to 0$ as $b \to 0$.
We can find $\tilde R$ from the above equations. From the third equation in (\ref{e15})
\begin{align}
    \begin{split}
        \tilde R = \zeta_\xi \ddot\xi + \zeta_\eta\ddot\eta + \zeta_{\xi\xi}\dot\xi^2 + 2 \zeta_{\xi\eta}\dot\xi\dot\eta + \zeta_{\eta\eta}\dot\eta^2 - \Phi_\zeta + mg(e_z,e_\zeta)
    \end{split}
\end{align}
and then substituting $\ddot\xi$ and $\ddot\eta$ from (\ref{e15}), we obtain
\begin{align}
\label{e17}
    \begin{split}
        \tilde R = \frac{1}{1 + \zeta_\xi^2 + \zeta_\eta^2} &\Big( \zeta_\xi [-mg (e_z, e_\xi) + \Phi_\xi] + \zeta_\eta [-mg (e_z, e_\eta) + \Phi_\eta] +\\
        &\zeta_{\xi\xi}\dot\xi^2 + 2 \zeta_{\xi\eta}\dot\xi\dot\eta + \zeta_{\eta\eta}\dot\eta^2 - \Phi_\zeta + mg(e_z,e_\zeta)\Big).
    \end{split}
\end{align}
For $\ddot f$ we have
\begin{align}
    \begin{split}
        \ddot f = &(e_z, e_\xi)\ddot\xi + (e_z,e_\eta)\ddot\eta + (e_z, e_\zeta)\ddot\zeta +\\
        &2(e_z, [\omega, e_\xi])\dot\xi + 2(e_z, [\omega,e_\eta])\dot\eta + 2(e_z, [\omega, e_\zeta]) \dot\zeta +\\
        &(e_z, [\dot\omega, e_\xi])\xi + (e_z, [\dot\omega,e_\eta])\eta + (e_z, [\dot\omega, e_\zeta]) \zeta +\\
        &(e_z, [\omega, [\omega, e_\xi]])\xi + (e_z, [\omega,[\omega,e_\eta]])\eta + (e_z, [\omega, [\omega, e_\zeta]]) \zeta
    \end{split}
\end{align}
Values of $|\dot\rho|$, $|\omega|$ and $|\dot\omega|$ can be made arbitrarily small by choosing $b$ and $c$. Hence, the sign of $\ddot f$ is governed by the following term
$$
-mg - \zeta_\xi(e_z,e_\xi) \tilde R - \zeta_\eta(e_z,e_\eta) \tilde R + (e_z, e_\zeta) \tilde R
$$
From geometrical considerations, vectors $(e_z,e_\zeta)e_\zeta + (e_z, e_\eta) e_\eta  + (e_z, e_\zeta) e_\zeta$ and $-\zeta_\xi e_\xi - \zeta_\eta e_\eta + e_\zeta$ cannot be parallel. Therefore, from (\ref{e17}), $|\tilde R| < mg/(1+\zeta_\xi^2 + \zeta_\eta^2)^{1/2}$ for small $b$ and $c$. Finally,
$$
-mg - \zeta_\xi(e_z,e_\xi) \tilde R - \zeta_\eta(e_z,e_\eta) \tilde R + (e_z, e_\zeta) \tilde R < 0.
$$
\end{proof}
From Lemmas 3.1 and 3.2 applying the Wa{\.z}ewski principle, we have
\begin{theorem}
There exist $b$ and $\mu$ such that for any $|\omega(t)| < b$ and $|\dot \omega(t)|<b$ for all $t$ system (\ref{e9}) has a solution $\rho(t)$ such that along this solution $f > 0$ for all $t$.
\end{theorem}
\begin{proof}
First, let $b$ and $c$ from Lemma 3.2 be chosen in such a way that for any point at $f = 0 \cap s = 0$, the set $T \leqslant c \cap \dot f \leqslant 0$ of the tangent space is not empty and is homeomorphic to a disk. It is possible to choose such $b$ and $c$ since set $\dot f \leqslant 0$ can be arbitrarily close to a half space of the tangent space when $b$ is small. Indeed, for $\omega = 0$, equation (\ref{e13}) has the form
\begin{align*}
\begin{split}
    \dot f = (e_z, e_\xi) \dot \xi + (e_z, e_\eta) \dot \eta + (e_z, e_\zeta) \dot \zeta.
\end{split}
\end{align*}
And $\dot f \leqslant 0$ defines a half space. Second, let us choose $\mu > 0$ such that the statement of Lemma 3.1 holds for the chosen $c$. Third, let us consider the following subset of the extended phase space:
\begin{align}
    N_c = \{ t, \rho, \dot\rho \colon f(t, \rho) \geqslant 0,\, s(\rho) = 0,\, T(\rho, \dot \rho) \leqslant c \}.
\end{align}
Any solution leaving this set can leave it only through the points of the following set
\begin{align}
    N_c^{++} = \{ t, \rho, \dot\rho \colon f = 0,\, s = 0,\, \dot f \leqslant 0,\, T \leqslant c \}.
\end{align}
Indeed, the solution cannot leave $N_c$ through the points where $T = c$ (unless we have $f = 0$ and the solution leaves $N_c$ through this part of the boundary). If $f = 0$ and $\dot f > 0$, the solution cannot leave the set either.

Let us consider a subset $D \subset N_c $ with the following properties:
\begin{enumerate}
    \item $D$ is homeomorphic to a disk.
    \item $\partial D \subset N_c^{++}$ and $D \cap \partial N_c = \partial D$.
    \item $D \subset \{ t, \rho, \dot \rho \colon t = 0 \}$.
\end{enumerate}
Let us show that there exists a point at $D$ such that the solution starting at this point always remains in $N_c \setminus \partial N_c$. Again, suppose the contrary. Then we can construct a retract between $D$ and its boundary (similarly to the case of pendulum).
\end{proof}

The key observation in the proof is similar to one for the pendulum. We show that any solution, tangent to the boundary of some region in the configuration space, locally leaves this region. In the case of pendulum this region is the positions of the rod where $(\rho, e_z) \geqslant 0$ (fixed region). For the surface this region is moving: $f(t, \rho) \geqslant$. In the both cases, we show that the gravity force is stronger than the vertical components of other forces, hence, the solutions leaves our regions.

Similarly, using results from \cite{srzednicki1994periodic,srzednicki2005fixed}, it is possible to show that the following also holds.
\begin{theorem}
There exist $b$ and $\mu$ such that for any $\tau$-periodic rotation of the surface satisfying $|\omega(t)| < b$ and $|\dot \omega(t)|<b$ for all $t$ system (\ref{e9}) has a $\tau$-periodic solution $\rho(t)$ such that along this solution $f > 0$ for all $t$.
\end{theorem}

The general method of the proof is the same as for the pendulum. The only difference here is that for the rotating surface the sections of $N_c$ by the planes $t = t_0$ for various $t_0$ do not coincide. Schematically this situation is illustrated in Fig. 4. However, in this case it is also possible to prove the existence of a periodic solution. The situation is similar to the following case. Let us again have system
\begin{align*}
\begin{split}
    & \dot x = v(x, y, t),\\
    & \dot y = w(x, y, t).
\end{split}
\end{align*}
And suppose that $N = \{ x, y, t \colon (x - \xi(t))^2 + (y - \eta(t))^2 \leqslant 1 \}$ where $\xi(t)$ and $\eta(t)$ are $\tau$-periodic functions (Fig. 4). Then there exists a $\tau$-periodic solution inside $N$ provided all trajectories in the extended phase space are transverse to the boundary of the deformed cylinder and all these solutions leave $N$.

\section{Conclusion}
We have considered two mechanical systems with gyroscopic forces and friction and shown that there exists a solution that always stays in some region. Moreover, this solution can be periodic provided the forces acting on the system are periodic.

The presented approach can be easily applied to various mechanical systems. For instance, one can consider the system of a point moving on a rotating surface without the assumption of the convexity of the surface. It is also possible to consider feedback control systems.

Note, that there exists a series of paper on the existence of periodic solutions in  Hamiltonian systems with magnetic or gyroscopic forces (indeed, frictionless) \cite{abbondandolo2014infinitely,taimanov1991nonselfintersecting,taimanov1992closed,kozlov1985calculus,bangert2010existence,ginzburg1996closed,ginzburg1999periodic,schneider2008closed}. In this regard, it is important to mention on the development a many-valued version of Morse theory \cite{novikov1981variational,novikov1981periodic,novikov1982hamiltonian}. At the same time, for many real-life mechanical systems it would be too optimistic to consider them as absolutely frictionless systems. However, the root of these problems is often purely mathematical and stems from the same problems for geodesics on a closed Riemannian manifold. In particular, the systems are not supposed to be non-autonomous. Results for non-autonomous systems can be found in \cite{furi1991forced,benci1986periodic}.

The case of non-autonomous systems with gyroscopic forces is far less developed. In particular, as far as we know, the are no results on the existence of forced oscillations for systems with gyroscopic forces and friction defined on a manifold with a boundary.

\section{Acknowledgement}
The research was funded by a grant from the Russian Science Foundation (Project No. 19-71-30012).












\section*{References}

\bibliographystyle{model1-num-names}
\bibliography{sample}







\end{document}